\documentclass[12pt]{siamonline1116}
\usepackage{graphicx}
\usepackage[margin=1in]{geometry}
\usepackage{amssymb}
\usepackage{amsmath}
\usepackage{tikz}
\usepackage[version=4]{mhchem}
\usepackage{hyperref}
\usepackage{enumitem}
\usepackage{rotating}

\usepackage{subfiles}

\newcommand{\Z}{\mathbb Z}

\newcommand{\Q}{\mathbb Q}

\newcommand{\R}{\mathbb R}

\newcommand{\gal}{\textrm{Gal}}

\newcommand{\rel}{\mathord{\sim}}
\colorlet{rxngreen}{green!80!black}

\newtheorem*{nota}{Notation}

\newsiamremark{remark}{Remark}
\newsiamremark{example}{Example}

\newlength\mylen
\newlist{pfcases}{enumerate}{1}
\setlist[pfcases,1]{label=\textbf{Case~\arabic*.}, 
  labelwidth=\dimexpr-\mylen-\labelsep\relax,leftmargin=0pt,align=right}

\title{Conditions for Solvability in Chemical Reaction Networks at Quasi-Steady-State}
\author{Mark Sweeney}
\date{\today}
\usepackage{enumitem}

\author{Mark A. Sweeney\thanks{Email: \email{msween14@rochester.edu}.}}

\begin{document}
\maketitle

\begin{abstract}
The quasi-steady-state assumption (QSSA) is an approximation that is widely used in chemistry and chemical engineering to simplify reaction mechanisms. The key step in the method requires a solution by radicals of a system of multivariate polynomials. However, Pantea, Gupta, Rawlings, and Craciun showed that there exist mechanisms for which the associated polynomials are not solvable by radicals, and hence reduction by QSSA is not possible. In practice, however, reduction by QSSA always succeeds. To provide some explanation for this phenomenon, we prove that solvability is guaranteed for a class of common chemical reaction networks. In the course of establishing this result, we examine the question of when it is possible to ensure that there are finitely many (quasi) steady states. We also apply our results to several examples, in particular demonstrating the minimality of the nonsolvable example presented by Pantea, Gupta, Rawlings, and Craciun.
\end{abstract}



\section{Introduction}
Many important reactions in chemistry are described by large reaction mechanisms, collections of small intermediate reactions which combine to form the overall reaction. These reactions give rise to a system of differential equations describing the change of concentrations over time, allowing the behavior of the system to be predicted and analyzed. Due to their complexity, it is often useful to simplify the system of differential equations so that calculations can be performed more quickly. One of the most well-known methods for reducing mechanisms is the quasi-steady-state assumption (QSSA), which replaces some of the differential equations with polynomial equations (in several variables). After solving these equations, the concentrations of some species can be removed from the model by substitution.

However, reduction by QSSA is only possible when the system of polynomials obtained from applying the assumption admits a finite number of solutions, all expressible in radicals. This is often the case in practice, and it was noticed only recently that there are mechanisms for which this reduction is not possible. In 2014, Pantea, Gupta, Rawlings, and Craciun presented two mechanisms for which reduction by QSSA cannot be performed~\cite{pantea}. One mechanism was constructed, but the other is an actual mechanism arising from chemistry, and one which would typically lend itself to reduction by QSSA.

QSSA is still taught today, and remains one of the most fundamental and simple methods of model reduction (\cite{fogler}, as PSSH). Given the widespread use of the method, we are interested in knowing when we can guarantee that reduction by QSSA is possible. In this work, we present conditions on the kinetics of the system and the structure of a graph derived from the mechanism which guarantee that QSSA can be performed. Also, these conditions are sufficient to demonstrate the minimality of the examples presented by Pantea, Gupta, Rawlings, and Craciun with respect to kinetics, number of intermediates, number of species, and a particular structural property of the network. The class of CRNs and intermediates discussed in this paper are fairly common in chemistry, providing a partial explanation for why, historically, the issue of insolvability was not noticed.

The paper is organized as follows. In Section 2, notation and definitions are introduced, as well as a description of the QSSA method. In Section 3, we recall several basic facts from commutative algebra, and then apply them to show that, under some mild conditions, reduction with QSSA is possible for networks with few intermediates. Subsequently, we show that if reduction is ``locally'' possible, then it can be extended to the entire network if the network satisfies certain structural and algebraic properties. A series of illustrative examples are provided in Section 4. We conclude the paper with a discussion of our results in a historical context, and describe further problems suggested by this work.


\section{Background}
\subsection{Chemical Reaction Networks}
First we state several important definitions used for describing chemical reaction networks (CRNs), following the presentation in~\cite{feinberg}, and the construction of the oriented species-reaction graph~\cite{osr}.

\begin{definition}
	A \textbf{chemical reaction network (CRN)} is described by the following information:
		\begin{enumerate}[leftmargin=2cm]
		\item A finite set of species, denoted by $\mathcal S$.
		\item A finite set of complexes, denoted by $\mathcal C$. These are formal linear combinations of species with coefficients in $\Z_{\geq 0}$. They can be identified with tuples from $\Z_{\geq 0}^{|\mathcal S|}$ indexed by $\mathcal S$.
		\item A finite set of reactions, denoted by $\mathcal R$, which can formally be viewed as a subset of $\mathcal C \times \mathcal C$, where a reaction $c \rightarrow c'$ is identified with the pair $(c,c')$.
		\end{enumerate}
\end{definition}
		
	Treating the complexes as vectors, we use $c_s$ to denote the coefficient of species $s$ in complex $c$. In more general circumstances, complexes may have coefficients in $\R$, but our results will not apply to such networks, as they do not result in polynomials after applying the QSSA.
	
	Reversible reactions are split into separate forward and reverse reactions.
	
	When we would like to emphasize that the concentration of a species $s$ is a variable, we will write $x_s$ for its concentration. When the concentration is treated as a parameter or constant, we typically write $c_s$ instead.

\begin{definition}
	The complex on the left in a reaction $c\rightarrow c'$ is referred to as the \textbf{reactants}, and that on the right as the \textbf{products}.
\end{definition}

\begin{definition}
	The \textbf{rate law polynomial} for a species $s$ in a CRN $(\mathcal S,\mathcal C,\mathcal R)$ is the derivative of its concentration with respect to time, and hence dictates the evolution of its concentration over time. For mass-action kinetics, it is defined as
	$$\frac{dx_s}{dt} = \Phi_s(\mathbf x) = \sum_{c\rightarrow c' \in \mathcal R} (c'-c)_A k_{c\rightarrow c'} \mathbf x^c.$$
	
	The vector $\mathbf x$ encodes all the species concentrations, indexed by $\mathcal S$. The term $\mathbf x^c$ is the product of each nonzero concentration in $\mathbf x$ raised to its coefficient in $c$.

	The coefficient $k_{c\rightarrow c'}$ is the \textbf{rate constant} of the reaction $c\rightarrow c'$. Because $c \in \Z^{|\mathcal S|}_{\geq 0}$, the expression $\mathbf s ^c$ is a monomial in the concentration variables.
\end{definition}

The set of all the rate law polynomials associated to a CRN is a system of differential equations which describes the evolution of all the concentrations in time.

\begin{definition}
	A $\Q$-linear dependence among the rate law polynomials of a CRN is called a \textbf{linear conservation law} (LCL).
	
	This means that, for each $s \in \mathcal S$, there exists some rational number $a_s$ such that:
	$$0 = \sum_{s\in\mathcal S} a_s \Phi_s(\mathbf x) = \sum_{s \in \mathcal S} a_s\frac{dx_s}{dt} = \frac d {dt} \sum_{s\in \mathcal S} a_s x_s,$$
	which is equivalent to saying
	$$\sum_{s \in \mathcal S} a_s x_s = T,$$
where $T$ is a constant determined by the initial conditions as
	$$T = \sum_{s\in\mathcal S} a_s x_{s}(t=0).$$

	For a subset of species $\mathcal Q$, we define $LCL(\mathcal R, \mathcal Q)$ as a set of conservation equations corresponding to a maximal set of independent LCLs arising from the rate law polynomials for each species $q \in \mathcal Q$.
\end{definition}

	Often, linear conservation laws describe constant total amounts of some quantity in the system. For instance, in many CRNs arising from biology, there is typically an LCL which corresponds to the total amount of each enzyme is fixed. For instance, Figure \ref{fig:intro} is a common model for enzyme-substrate reactions. Examining the differential equations, we can find one rate law:
	$$\frac{dx_E}{dt} + \frac{x_{E\cdot S}}{dt} = (-k_1x_Ex_S + k_{-1}x_{E\cdot S} + k_2x_{E\cdot S}) + (k_1x_Ex_S - k_{-1}x_{E\cdot S} - k_2x_{E\cdot S}) = 0$$
	
	Physically we recognize that this sum is the total amount of enzyme in the system, so the LCL would be of the form $x_E + x_{E\cdot S} = T$, where $T$ is the amount of enzyme.

\begin{definition}
	A CRN is \textbf{at-most-bimolecular} if the reactants are always of the form $2A$ or $A + B$. This implies that each rate law polynomials is at most degree $2$ in the concentrations.
\end{definition}

Such networks are common in chemistry: for an ``elementary'' mechanism, each reaction corresponds to the collision and subsequent reaction of the reactants. The probability of a 3-way collision occurring is so small that such reactions are not common in practice, and proceed very slowly when they do occur~\cite{fogler}.

Next, we define an associated graph, whose structure is closely related to the structure of the differential equations. The definition here largely coincides with that in~\cite{osr}, but here reversible reactions are treated as two separate reactions.

\begin{definition}
	The \textbf{oriented species-reaction graph (OSR graph)} is defined as follows. 

	The vertices of the graph are given by the (disjoint) union of $\mathcal S$ and $\mathcal R$. For each reaction $c\rightarrow c'$, an edge is drawn from $s$ to $c\rightarrow c'$ if $c_s$ is nonzero, and an edge from $c\rightarrow c'$ to $s$ if $c'_s$ is nonzero. Edges are weighted by the value of $c_s$ or $c'_s$, respectively.
\end{definition}

	When drawn, species are generally enclosed in circles, and reactions in rectangles. A small example is depicted in Figure \ref{fig:intro}.

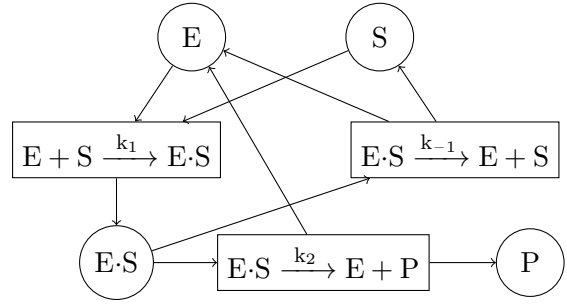
\begin{figure}
	\begin{center}
		\ce{E + S <=>[k_1][k_{-1}] E\cdot S ->[k_2] E + P}
	
		\vspace{0.5cm}

		\begin{minipage}{0.5\textwidth}
		\begin{align*}
			\frac{dx_E}{dt} &= -k_1x_Ex_S + k_{-1}x_{E\cdot S} + k_2 x_{E\cdot S}\\
			\frac{dx_S}{dt} &= -k_1x_Ex_S + k_{-1}x_{E\cdot S}\\
			\frac{dx_{E\cdot S}}{dt} &= k_1x_Ex_S - k_{-1}x_{E\cdot S} - k_2 x_{E\cdot S}\\
			\frac{dx_P}{dt} &= k_2 x_{E\cdot S}
		\end{align*}
		
		\end{minipage}
		\begin{minipage}{0.45\textwidth}
			\begin{tikzpicture}
				\node[circle,draw,minimum size=.9cm] at (-1,3) (E) {E};
				\node[circle,draw,minimum size=.9cm] at (3.5,0) (P) {P};
				\node[circle,draw,minimum size=.9cm] at (1.5,3) (S) {S};
				\node[circle,draw,minimum size=.9cm] at (-2,0) (ES) {E$\cdot$S};
	
				\node[rectangle,draw] at (-2,1.5) (r1) {\ce{E + S ->[k_1] E$\cdot$S}};
				\node[rectangle,draw] at (2.5,1.5) (r-1) {\ce{E$\cdot$S ->[k_{-1}] E + S}};
				\node[rectangle,draw] at (0.75,0) (r2) {\ce{E$\cdot$S ->[k_2] E + P}};
	
				\draw[->] (E) -- (r1);
				\draw[->] (S) -- (r1);
				\draw[->] (ES) -- (r-1);
				\draw[->] (ES) -- (r2);
	
				\draw[->] (r1) -- (ES);
				\draw[->] (r-1) -- (E);
				\draw[->] (r-1) -- (S);
				\draw[->] (r2) -- (P);
				\draw[->] (r2) -- (E);
			\end{tikzpicture}
		\end{minipage}
	\end{center}
\caption{\label{fig:intro}Example of a chemical reaction network, the corresponding differential equations, and the OSR graph.}
\end{figure}

\subsection{Quasi-Steady-State Assumption}
As discussed in the introduction, the quasi-steady-state assumption (QSSA) is a way of reducing a system of differential equations. The method is as follows:
\begin{enumerate}[leftmargin=2cm]
	\item Identify a subset of species $\mathcal Q \subseteq S$ as intermediates.
	\item Set the corresponding rate law polynomials, $\{\Phi_q\ |\ q \in\mathcal Q\}$, equal to zero (this is the application of the QSSA).
	\item From these polynomial equations, $\{\Phi_q = 0\ |\ q\in\mathcal Q\}$, solve for the concentrations $x_q$, where $q\in\mathcal Q$, in terms of the rate constants and non-intermediate concentrations.
	\item Substitute these solutions into the remaining differential equations, eliminating the concentrations of the species in $\mathcal Q$ from the model.
\end{enumerate}

For example, consider the mechanism depicted in Figure \ref{fig:intro}. The enzyme-substrate complex $E\cdot S$ may be an unstable structure, and hence a good candidate for elimination by QSSA. Apply the QSSA by setting the corresponding rate law polynomial equal to zero:
$$\frac{d[E\cdot S]}{dt} = k_1[E][S] - k_{-1}[E\cdot S] - k_2 [E\cdot S] = 0$$
$$\Rightarrow [E\cdot S] = \frac{k_1}{k_{-1} + k_2}[E][S]$$

This expression for $[E\cdot S]$ can be substituted into the remaining differential equations, eliminating the intermediate's concentration from them.

When reduction by QSSA is possible, this is a powerful technique for simplifying the model, and one that is particularly useful in applications. For instance, most reasonable intermediates exist at very small concentrations, and so eliminating them removes variables which are difficult to measure experimentally.

Reduction by QSSA is most useful when all of the solutions from the third step can be expressible by radicals. This is typically the case, especially for systems which are analyzed manually. However, Pantea, Gupta, Rawlings, and Craciun found two mechanisms for which they proved that there is no solutions expressible in radicals~\cite{pantea}. This naturally leads us to ask for which CRNs we can guarantee that reduction by QSSA is or is not possible.

\subsubsection{Definitions}
Before presenting the main results, we need a few further definitions.

Let $(\mathcal S,\mathcal C,\mathcal R)$ be a CRN and $\mathcal Q \subseteq \mathcal S$ a subset of species identified as intermediates to eliminate by QSSA.

\begin{definition}
	For QSSA, solvability is determined with respect to a particular ground field, $\Bbbk_{\mathcal Q}$. This field is obtained by adjoining to $\Q$ the nonintermediate concentrations, rate constants, and constant terms resulting from linear concentration laws:
	$$\Bbbk_{\mathcal Q} = \Q(\{c_s\ |\ s \in \mathcal S - \mathcal Q\}\cup\{k_{c\rightarrow c'}\ |\ c\rightarrow c' \in \mathcal R\}\cup \{T\ |\ T\textrm{ a constant term of }f\in LCL(\mathcal R,\mathcal Q)\})$$
\end{definition}

	The $c$'s, $k$'s, and $T$'s correspond to the concentrations of the nonintermediate species, rate constants, and conservation law parameters, respectively. To simplify the notation, these parameters are often indexed by integers. All of these parameters are algebraically independent.
	
	The subscript on $\Bbbk_{\mathcal Q}$ may be omitted when it is clear from context.

\begin{definition}
	Computations will take place in one of two polynomial rings: $A_{\mathcal Q} = \Bbbk[x_q\ |\ q \in \mathcal Q]$ or $B_{\mathcal Q} = \Bbbk^a[x_q\ |\ q\in\mathcal Q]$, where $\Bbbk^a$ denotes the algebraic closure. 
\end{definition}

To perform QSSA, we are interested in solving a system of polynomial equations from $A_{\mathcal Q}$. Zeros of systems of equations correspond to ideals, motivating the following:

\begin{definition}
The \textbf{quasi-steady-state ideal}, denoted by $I_{\mathcal Q}$, is the ideal generated by the rate law polynomials for each intermediate, along with any conservation laws:
	$$I_{\mathcal Q} = \langle \{\Phi_q\ |\ q \in \mathcal Q\} \cup LCL(\mathcal R,\mathcal Q)\rangle$$
\end{definition}
	
	We can view $I_{\mathcal Q}$ as an ideal of either $A_{\mathcal Q}$ or $B_{\mathcal Q}$. Where this is not clear from the context, it will be specified.

\begin{nota}
	For an ideal $I$ in $k[x_1,...,x_n]$, we use $V(I)$ to denote the set of points in $k^n$ on which every polynomial in $I$ vanishes, and $V^a(I)$ when we take the points from $(k^a)^n$.
	
	For the quasi-steady-state ideal $I_{\mathcal Q}$, these zeros may be called either solutions or quasi-steady-states.
\end{nota}

\begin{definition}
	We say that a solution $\alpha \in V^a(I_{\mathcal Q})$ is a \textbf{boundary solution} if at least one of its coordinates is zero.
\end{definition}

	The presence of a boundary quasi-steady-state may be an indicator that a poor choice of intermediates was made, as it means that this particular quasi-steady-state requires some intermediate to be absent from the system. However, it could be the case that even a reasonable choice of intermediates leads to some spurious boundary solutions.

\begin{remark}
	Under our definition of QSSA, the LCLs are included as constraints in the system. In the definition of $I_{\mathcal Q}$, we restricted to LCLs involving only the intermediates: this is because the QSSA implies the intermediates react on a faster time scale than the slow species. Essentially, the derivative with respect to $t$ (time) becomes a derivative with respect to some $t'$, which means we can no longer pull it out of the sum in an LCL.
	
There are both advantages and disadvantages to treating systems which require LCLs to make QSSA possible. On one hand, allowing LCLs lets us treat a wider range of CRNs. On the other, the constants associated to the LCLs can be difficult or impossible to measure in practice. Further, species for which there exist LCLs are more likely to be poor candidates for elimination by QSSA.

As such, it may be preferable to say that QSSA reduction is impossible for any collection of intermediates for which there are LCLs. The only results in this work which depend on the inclusion of the LCLs are those about the number of (quasi) steady states, as it may be necessary to include them to guarantee that there are finitely many steady states. Naturally, all of the results presented here remain true if we restrict QSSA to a choice of intermediates for which there are no LCLs.
\end{remark}

The following example illustrates these definitions:

\begin{example}
	Consider the following CRN:
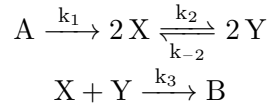
\begin{center}
	\ce{A ->[k_1] 2X  <=>[k_2][k_{-2}] 2Y}
	
	\ce{X + Y ->[k_3] B}
\end{center}

Let $\mathcal Q =\{X,Y\}$ be the intermediates. For simplicity, let $a=[A]$, $b=[B]$, $x=[X]$, and $y=[Y]$. There are no LCLs, so the ground field is just
$$\Bbbk_{\mathcal Q} = \Q(a,b,k_1,k_2,k_{-2},k_3).$$

This means $A_{\mathcal Q} = \Bbbk_{\mathcal Q} [x,y]$. The ideal $I_{\mathcal Q}$ is generated by the following two polynomials:
\begin{align*}
	\Phi_X(x,y) &= -2k_2x^2 - k_3xy + 2k_{-2}y^2 + k_1a\\
\Phi_Y(x,y) &= -2k_{-2}y^2 - k_3xy + 2k_2x^2.
\end{align*}

A Gr\"obner basis for $I_{\mathcal Q}$ over $\Bbbk_{\mathcal Q}$ contains the univariate polynomial
\begin{align*}
	f(x) =& (8k_{-2}k_2^2-3k_2k_3^2)x^4+(8k_{-2}k_2k_3)x^3\\
	      &+(-8ak_{-2}k_1k_2+ak_1k_3^2-4k_{-2}^2k_2)x^2\\
	      &-(2k_{-2}k_1ak_3)x+(2a^2k_{-2}k_1^2).
\end{align*}

Thus, the $x$-coordinate of any quasi-steady-state must be a root of this polynomial. The Galois group of this polynomial over $\Bbbk$ is isomorphic to $D_8$, which is solvable, and so the $x$ component of any solution can be expressed in radicals. The Galois group in the $y$-coordinate is also isomorphic to $D_8$, so reduction by QSSA is possible for this CRN and choice of intermediates.

It is somewhat surprising that the Galois group of $f$ is not just isomorphic to $S_4$, as $f$ does not seem to be particularly structured. Indeed, the relationship between the mechanism and the Galois groups is not well-understood, but the simplicity of the Galois group may be a reflection of the simplicity of the original mechanism.
\end{example}

We are interested in characterizing CRNs and sets of intermediates for which reduction by QSSA is possible. By ``possible'' we mean that:
\begin{enumerate}[leftmargin=2cm]
	\item There are finitely many solutions to the corresponding system of polynomials, that is, $V^a(I_{\mathcal Q})$ is finite, and
	\item These solutions can be obtained by radicals: for each point $(\alpha_1,...,\alpha_m) \in V^a(I_{\mathcal Q})$, the extension $\Bbbk(\alpha_i)$ is solvable.
\end{enumerate}

This is not the only possible definition. As discussed previously, it might be reasonable to require that there are no LCLs. We could also weaken the definition to only require finiteness and solvability for real solutions. This can be further weakened to nonzero, nonnegative, or positive concentrations, which are, after all, the physically meaningful concentrations.


\section{Main Results}
This section presents the main results on CRN rate law polynomials. First, we discuss some useful properties of the coefficients of the rate law polynomials $(\Phi_s)$, some of which address the role LCLs play in reduction by QSSA. Next, we present simple solvability criteria for the rate law polynomials of CRNs, as well as a useful finiteness condition. This section concludes with a graph-theoretical construction which allows the previous results to be extended to larger networks.

\subsection{Properties of Rate Law Polynomials}
As we shall see later, the finiteness of $V^a(I_{\mathcal Q})$ is crucial. Not only is it necessary to perform QSSA, it is also the first thing which must be established before proving solvability results. For most choices of intermediates, the ideal $I_{\mathcal Q}$ has finitely many zeros. Intuitively, we expect this to be the case, as the algebraically independent rate constants appearing in the coefficients make the rate law polynomials behave almost as if they have generic coefficients. Furthermore, one common way in which rate law polynomials exhibit degenerate behavior implies the existence of an LCL, which corrects for the degeneracy.

As a matter of notation, when we write that a value depends on some collection of rate constants $\{k_1,...,k_n\}$ we mean that it is algebraic over $\Q(k_1,...,k_n)$.

\begin{proposition}
	The coefficients $c_\alpha$ and $c_\beta$ of distinct monomials $x^\alpha$ and $x^\beta$ of rate law polynomials depend on disjoint sets of rate constants.
\end{proposition}

This is a well-known and simple result about rate law polynomials, but it will prove to be crucial in what follows. It suggests that the coefficients of rate law polynomials are almost algebraically independent, and so it is particularly useful for showing that identities relating the coefficients of the rate law polynomials cannot exist.

For instance, consider two rate law polynomials arising from the same network:
\begin{align*}
	\Phi_x(x,y) &= ax  + by  + c\\
	\Phi_y(x,y) &= a'x + b'y + c'.
\end{align*}

Assume that $\Phi_x$ is nonconstant, and that it is possible to form some identity from these coefficients, such as
$$b + bc' - a^2 = 0$$

Because $c'$ is algebraically independent of $a$ and $b$, it must be the case that $b=0$ or $c' = 0$. In the second case, we obtain $b-a^2 = 0$, which requires $b=0$ as well. However, if $b=0$, then it must be that $a=0$ as well, since the identity becomes $a^2 = 0$. This contradicts the claim that $\Phi_x$ is nonconstant.

We can exploit a similar argument to show that rate law polynomials are well-behaved with respect to divisibility. We will show that if one rate law polynomial is a multiple of another over $\Bbbk$, then they are in fact multiples over $\Q$, assuming they are not both monomials. When analyzing QSSA, this provides a workaround for some degeneracies.

\begin{lemma}
	Suppose that $\alpha$ is algebraic over both $\Q(x_1,...,x_n)$ and $\Q(y_1,...,y_m)$, where the $x_i$s and $y_j$s are algebraically independent over $\Q$. Then $\alpha$ is in fact algebraic over $\Q$.
	\label{inddia}
\end{lemma}
\begin{proof}
	Let $k=\Q(x_1,...,x_n)$, and $L$ the normal closure of $k(\alpha)$. We can see that $[L(y_1,...,y_m):k(y_1,...,y_m)] = [L:k]$ because of the algebraic independence of the $x$ and $y$ terms. This means that every conjugate of $\alpha$ over $k$ is also a conjugate of $\alpha$ over $L(y_1,...,y_m)$. But this analysis remains true if we start with $k = \Q(y_1,...,y_m)$ instead. As a result, the conjugates of $\alpha$ are the same over both $\Q(x_1,...,x_n)$ and $\Q(y_1,...,y_m)$. This, however, means that $\alpha$ has the same minimal polynomial over both fields, and hence that the minimal polynomial has coefficients from their intersection, which is $\Q$.
\end{proof}

\begin{proposition}
	\label{indprop}
	Suppose that $f$ and $g$ are polynomials over $\Bbbk = \Q(k_1,...,k_m)$ which satisfy the following:
	\begin{enumerate}[leftmargin=2cm]
		\item $f=\alpha g$ for some $\alpha \in \Bbbk$,
		\item coefficients of distinct monomials taken from either $f$ or $g$ depend on distinct subsets of $\{k_1,...,k_m\}$,
		\item $f$ and $g$ each have at least two nonzero terms.
	\end{enumerate}
	
	Then $\alpha$ in $\Q$.
	\label{coeffprop}
\end{proposition}
\begin{proof}
	The second condition implies that $a,a'$ are algebraic over $\Q(S)$ and $b,b'$ over $\Q(R)$ for some disjoint subsets $S$ and $R$ of $\{k_1,...,k_m\}$.
	
	Let the coefficients of the corresponding nonzero terms be $a,b$ and $a',b'$, respectively. We can compare these terms according to the first condition,
	$$a= \alpha a', b = \alpha b'$$
	$$ab' = ba'$$
	
	But this means that $\alpha$ is algebraic over both $\Q(S)$ and $\Q(R)$. By applying Lemma \ref{inddia}, we see that $\alpha$ is actually algebraic over $\Q$. Furthermore, $\alpha \in \Q(k_1,...,k_m)$ which contains no proper algebraic extension of $\Q$, and so $\alpha$ must be an element of $\Q$ itself.
\end{proof}

\begin{corollary}
	Let $\Phi_r$ and $\Phi_s$ be two rate law polynomials from a CRN, viewed over $\Bbbk_{\mathcal Q}$ for some choice of intermediates $\mathcal Q$. If $\Phi_r$ divides $\Phi_s$, $\deg(\Phi_r) =\deg(\Phi_s)$, and both have at least two nonzero terms, then $\Phi_r = \alpha \Phi_s$ where $\alpha \in \Q$. In particular, there exists an LCL between species $r$ and $s$.
	\label{indcor}
\end{corollary}
\begin{proof}
	The coefficients of $f$ and $g$ lie in $\Bbbk_{\mathcal Q}$, and so the concentration parameters in this field prevent the immediate application of Proposition \ref{coeffprop}. However, this difficulty is relatively straightforward to overcome.
	
	 For each rate constant $k_i$, we replace products of the form $k_i c_1...c_m$ with a new parameter $\tilde k_i$. Because each rate constant is associated with a unique reaction, and each reaction with a particular set of reactants, this means that there are some fixed nonintermediate concentrations with which $k_i$ always appears, so this is really a linear substitution $k_i \mapsto \frac{\tilde k_i}{c_1...c_m}$. As such, it preserves the divisibility of $\Phi_r$ and $\Phi_s$ as well as the dependence properties of their coefficients.
	
	Following this substitution, we can apply Proposition \ref{coeffprop} over the field $\Q(\tilde k_1,...\tilde k_n)$ to see that $\alpha \in \Q$. Thus we obtain the $\Q$-linear dependence $\Phi_r - \alpha \Phi_s = 0$, which corresponds to the LCL $x_r - \alpha x_s = T$.
\end{proof}

Corollary \ref{indcor} shows that, even in this highly degenerate situation, the structure of the rate law polynomials ensures the existence of an LCL. This acts to compensate for the degeneracy by introducing a new constraint on the (quasi) steady states. It further highlights that even the algebraic independence of the nonintermediate concentrations is not important, so long as they are all nonzero.

\subsection{Specialization of Parameters}
In defining $\Bbbk$, we chose algebraically independent rate constants. Their independence is crucial to later results, but in practice some of their values may be known, and hence we need to verify that specializing some rate constants to certain values does not alter solvability. By specializing parameters, we mean setting them equal to some value in the sense of a substitution.

Pantea, Gupta, Rawlings, and Craciun remarked that it is clear that solvable solutions are preserved by specialization (in the sense that they can still be expressed in radicals over the ground field obtained after performing the specialization) because it is possible to just specialize in the general solution~\cite{pantea}. However, consider the following ``counterexample'':
$$f(x) = tx^2 + x + 1,$$
which has roots
$$\frac{-1 \pm \sqrt{1 - 4t}}{2t}.$$

We can make $f$ linear by setting $t=0$, but it is not possible to specialize the general solution without dividing by zero. While solvability is preserved in this case, we might worry that there exist other pathological specializations which could somehow cause solvability to be lost. It turns out this is not the case, as we will show below.

In one sense, knowing this fact is not critical, as QSSA is generally applied before the $k_i$ are known, and so in many situations no specialization is done while using QSSA. On the other hand, making simplifying substitutions can speed up computations. The more significant problem is that specialization will not necessarily preserve the finiteness of $V^a(I_{\mathcal Q})$, a problem which arises even in simple CRNs. This becomes important in later sections.

As we saw above, specialization may lead to division by zero. Furthermore, we run into the problem of specialization in an extension - we would like to be able to extend a specialization of the ground field to the whole extension in a ``natural way''. The following lemma addresses how this can be done and how it interacts with solvability.

\begin{lemma}
	Suppose $\mathbb L /\Bbbk$ is a Galois extension, where $\Bbbk = \Q(x_1,...,x_n,y_1,...,y_m)$ with the $x_i$ algebraically independent. If we specialize the $x_i$ parameters from $\Bbbk$ to obtain a field $k$, it is possible to extend the specialization to $\mathbb L$, resulting in some field $L$. If $\mathbb L/\Bbbk$ is solvable, so is $L/k$.
	\label{speclemma}
\end{lemma}
\begin{proof}
 For each $x_i$, let $t_i$ be the value to which it should be specialized. First, we lift $\mathbb L/\Bbbk$ to $\mathbb L(t_1,...,t_n)/\Bbbk(t_1,...,t_n)$. This lifting induces an embedding of $G = \gal(\mathbb L(t_1,...,t_n)/\Bbbk(t_1,...,t_n)$ into $\gal(\mathbb L/\Bbbk)$.
	
	Let $A=\mathbb Q(y_1,...,y_m)(t_1,...,t_n)[k_1,...,k_n]$ (i.e. the subring of $\Bbbk$ consisting of polynomials in the $k_i$), and $B$ its integral closure in $\mathbb L$. Let $\mathfrak p = \langle k_1-t_1,...,k_n-t_n\rangle$. Note that $\mathfrak p$ is maximal and that $A/\mathfrak p \cong k$. Let $\mathfrak P$ be a maximal ideal of $B$ lying over $\mathfrak p$. Then let $L = B/\mathfrak P$. By construction, $L/k$ is a field extension. We can think of it as extending the specialization from $\Bbbk$ to $\mathbb L$; since $\mathfrak P$ contains each $(x_i-t_i)$, we essentially identify $x_i$ with $t_i$ in the factor ring $B/\mathfrak P$. Furthermore, there is a surjective homomorphism from the decomposition group $G_{\mathfrak P}$ to $\gal(L/k)$ (Proposition 14,~\cite{antlang}). 
	
	Subgroups and homomorphic images of solvable groups are solvable. Since $G$ can be identified with a subgroup of the solvable group $\gal(\mathbb L/\Bbbk)$, it is solvable. Similarly, $G_{\mathfrak P}$ is solvable because it is a subgroup of $G$. Finally, the extension we are interested in is solvable because there is a surjection from $G_{\mathfrak P}$ onto it.
\end{proof}

Extending $\mathfrak p$ to $\mathfrak P$ extends the specialization from the ground field to the whole extension. While the choice of $\mathfrak P$ is noncanonical, the quotient field obtained is unique up to isomorphism.

When applying Lemma \ref{speclemma} in the context of QSSA, we think of the $x_i$ as rate constants ($k_i$) and the $y_i$ as the other terms adjoined to $\Bbbk_\mathcal Q$; namely, the nonintermediate concentrations ($c_i$) and constant terms of LCLs ($T$).

\begin{remark}
	From Lemma \ref{speclemma} it follows that specializing the rate constants does preserve solvability if finiteness is also preserved. We will later show that solvability can be determined from the solvability of polynomials obtained by a Gr\"obner basis calculation (Proposition \ref{solvprop}). So long as all of these polynomials exist, and are nonzero and solvable, reduction by QSSA is possible.
	
	Since the Gr\"obner basis calculations can be performed without division, they are preserved by specialization, which means we simply obtain the specialized versions of the initial polynomials - as we proved above, their solvability is preserved. However, it is entirely possible that one or more polynomials in the Gr\"obner basis becomes the zero polynomial after specializing.
\end{remark}

	Lemma \ref{speclemma} allows us to work with many CRNs at once, even those for which some rate constants are known, by treating the rate constants as parameters. Unfortunately, there can be no version of Lemma \ref{speclemma} for insolvable groups, as they are not closed under subgroups or homomorphic images. In particular, there can be specializations of the coefficients which send insolvable Galois groups to solvable Galois groups. This makes classifying  insolvable networks more difficult, as any such theory would have to also consider the possible specializations of the coefficients. Any complete classification of the networks for which QSSA is possible would also need to cover these cases, for which our techniques are not as useful.

\subsection{Simple Solvability Criteria}
While Galois theory provides a solvability criterion for single-variable polynomials, the polynomials obtained from applying the QSSA are rarely univariate, so a crucial first step is to reduce to that case. We begin by recalling several facts from commutative algebra.

\begin{proposition}[{{\cite{becker}}}]
	\label{finitevariety}
	Suppose $I$ is an ideal in $k[x_1,...,x_n]$, where $k$ is a field. Then $V^a(I)$ is finite (i.e. $I$ is a zero-dimensional ideal) if and only if $I \cap k[x_i] \neq \{0\}$ for each $1\leq i \leq n$.
\end{proposition}

For most CRNs arising from chemistry, there are only finitely many steady states when the LCLs are accounted for, and so $V^a(I_{\mathcal Q})$ is typically finite. When there are infinitely many solutions, QSSA will fail anyway, as this makes it impossible to select concentrations of the intermediates to substitute into the remaining equations. Furthermore, in such a situation some of the concentrations can be selected more or less freely at steady-state, which is not typical behavior for well-chosen intermediates.

Galois theory only addresses univariate polynomials, and it is possible to reduce to this case exactly when $V^a(I_{\mathcal Q})$ is finite. Passing to the single-variable case allows us to determine solvability by computing the Galois groups corresponding to generators of $I_{\mathcal Q} \cap k[x_j]$ over $k$. The computation of such a generator requires the following proposition.

\begin{proposition}
	Let $I$ be as in Proposition \ref{finitevariety}, and $G$ a reduced Gr\"obner basis for $I$ with respect to the the lexicographic ordering
	$$x_1 > x_2 > ... >  x_n.$$
	Then $I\cap k[x_n]$ is generated by $G\cap k[x_n]$, and this intersection contains only one element.
\label{elim}
\end{proposition}

Proposition \ref{elim} is a special case of the Elimination Theorem~\cite{iva}.  Since it is possible to compute Gr\"obner bases and Galois groups over $\Bbbk$, this suggests a very basic criterion to determine whether reduction by QSSA is possible.

\begin{proposition}
	QSSA reduction can be performed if and only if for each $q\in\mathcal Q$, $I_{\mathcal Q} \cap \Bbbk[x_q] \neq \{0\}$ and the generator of this intersection is solvable over $\Bbbk$. These conditions can be verified algorithmically.
	\label{algprop}
\end{proposition}
\begin{proof}
	As remarked earlier, QSSA reduction is not possible if $V^a(I_{\mathcal Q})$ is infinite. If it is finite, Proposition \ref{finitevariety} shows that $I_{\mathcal Q} \cap \Bbbk[x_i]$ is nontrivial for each variable $x_i$.
	
	$\Leftarrow$: Let $g_i$ be a generator of $I_{\mathcal Q} \cap \Bbbk[x_i]$ as an ideal of $\Bbbk[x_i]$. Then the zeros of $g_i$ are expressible in radicals if and only if the Galois group of $g_i$ is solvable over $\Bbbk$. Because $g_i \in I_{\mathcal Q}$, it vanishes everywhere on $V^a(I_{\mathcal Q})$, which means that $i$th coordinate of any zero of the ideal is a root $g_i$. Since $g_i$ is solvable, all of these roots must be expressible by radicals over $\Bbbk$.
	
	Conversely, if QSSA reduction is possible, then the generators of $I_{\mathcal Q}$ have finitely many zeros, all expressible by radicals. Since there are finitely many, Proposition \ref{finitevariety} ensures that $I_{\mathcal Q} \cap \Bbbk[x_q]$ is nonempty. In proving Proposition \ref{finitevariety}, the Nullstellensatz is used to show that a particular polynomial is in $I_{\mathcal Q} \cap k[x_i]$. If every point in the variety has solvable coordinates, this polynomial can be chosen such that it is solvable.
	
Let $\{x_1,...,x_n\}$ be an enumeration of the variables corresponding to the intermediate concentrations. The conditions in the statement of this proposition can be verified in two steps:
	\begin{enumerate}[leftmargin=2cm]
		\item For each $x_j$, compute a reduced Gr\"obner basis, $G$, of $I_{\mathcal Q}$ with respect to
		$$x_1 > ... > x_{j-1} > x_{j+1} > ... > x_n > x_j$$
		Let $g_j$ be the unique element of $G \cap k[x_j]$. If there is no such element for some $j$, reduction by QSSA is not possible (by Prop. \ref{finitevariety}, this would imply $V^a(I_{\mathcal Q})$ is not finite).
		
		\item Compute the Galois group of $g_j$ over $\Bbbk$, and test for solvability.
	\end{enumerate}
\end{proof}

While Proposition \ref{algprop} provides an algorithmic test for solvability, it is computationally expensive. This proposition computes $n$ Gr\"obner bases, the complexity of which is doubly exponential in the degree and number of variables. While some of our later results suggest that the computational complexity is smaller for polynomials arising from CRNs, it is preferable to avoid these computations when possible. Indeed, for many examples, the computation of the Galois group is the most time-consuming step; in one example, we did not find the Galois group even after several hours. Additionally, the algorithm available for performing this computations only works for polynomials of degree at most eight~\cite{galoisalg}. This is a further motivation to develop criteria for solvability based on the structure of the CRN, as we do in this work.

Recall that the Galois group of a polynomial $f$ of degree $n$ is determined by its action on the $n$ roots of $f$, and so it naturally embeds into $S_n$. Thus, if $n \leq 4$, the Galois group will be solvable, so a degree bound can guarantee solvability.

\begin{proposition}
	\label{solvprop}
	Consider a CRN with at-most-bimolecular kinetics, at most two intermediates, and for which $V^a(I_{\mathcal Q})$ is finite. Then it is possible to perform reduction with QSSA on this system.
\end{proposition}
\begin{proof}
	If there is just one intermediate, $I_{\mathcal Q}$ is generated by a (single) linear or quadratic univariate polynomial, which is certainly solvable.
	
	Now assume there are two intermediates, whose concentrations we will denote by $x_1$ and $x_2$. Here we work in $\Bbbk[x_1,x_2]$. We know that $I_{\mathcal Q}$ is generated by the two rate law polynomials, $\Phi_1$ and $\Phi_2$, and possibly some linear conservation laws. Due to the choice of kinetics, both $\Phi_1$ and $\Phi_2$ have degree at most two.
	
	By hypothesis, $I_{\mathcal Q}\cap \Bbbk[x_1]$ and $I_{\mathcal Q} \cap \Bbbk[x_2]$ are nonzero ideals, and thus have nonzero generators $g_1$ and $g_2$, respectively. It suffices to show both polynomials have solvable Galois groups. This we will do with a degree bound. From Bezout's theorem, we know that the degrees of $g_1$ and $g_2$ are bounded by product of the degrees of the generators of $I_{\mathcal Q}$ - this is at most 4, as $I_{\mathcal Q}$ has just two generators, neither of which have degree greater than two.
	
	Since this means that $g_1$ and $g_2$ each have degree at most $4$, both of their Galois groups must be solvable over $\Bbbk$.
\end{proof}

\subsection{Criterion for Finitely Many Quasi-Steady-States}
The results thus far highlight the importance of ensuring that the number of quasi-steady-states is finite when applying QSSA, which is equivalent to determining when $V^a(I_{\mathcal Q})$ is finite. This motivates the following theorem, which shows that the finiteness of $V^a(I_{\mathcal Q})$ is guaranteed for two intermediates under a few mild conditions.

Physically, we would not expect infinitely many (quasi) steady states, so it seems reasonable that Theorem \ref{finitethm} will generalize (perhaps with slightly different hypotheses). However, the method of proof does not appear to generalize, even to the case of three intermediates. On the other hand, this theorem shows that there are finitely many solutions in the algebraic closure, but CRNs only have solutions with nonnegative real concentrations. It is possible that there is a network with a choice of intermediates for which there are infinitely many solutions in $\Bbbk^a$ but only finitely many which are in $\R$ or $\R_{\geq 0}$.

As mentioned earlier, results like Theorem \ref{finitethm} below are useful from a computational perspective. It is known that the computational complexity of finding a Gr\"obner basis is lower for zero dimensional ideals~\cite{hashemi}, so this would indicate that calculating steady states arising from CRNs may be more tractable than finding Gr\"obner bases for more general ideals, though only when specialized to actual numbers, as the genericity of the coefficients increases the complexity. Since the algorithm outlined earlier requires computing $n$ Gr\"obner bases, it would be helpful to have a reasonable assurance that the Gr\"obner basis calculations will be quick, even if the Galois group is insolvable.

We begin with a theorem which shows that for most choices of intermediates, there are only finitely many nonboundary steady states. To state this algebraically, we require a few more constructions from commutative algebra:

\begin{definition}
	Let $I$ and $J$ be ideals in a ring $A$. Then the \textbf{ideal quotient} of $I$ and $J$, denoted $(I:J)$ is the following ideal:
	$$(I:J) = \{a \in A\ |\ aJ \subseteq I\}.$$
\end{definition}

\begin{definition}
	Let $I$ and $J$ be ideals in a ring $I$. The \textbf{saturation} of $I$ with respect to $J$, denoted $(I:J^\infty)$ is defined as:
	$$\bigcup_{n=1}^\infty (I:J^n).$$
\end{definition}

It can be shown that
$$(I:J) \subseteq (I:J^2) \subseteq ...$$

Which means that, in the case of a Noetherian ring, only finitely many of the $(I:J^k)$ need to be computed to determine $(I:J^\infty)$. Algorithms exist to perform this computation, so our results usuing these ideals are also constructive~\cite{iva,becker}.

Saturation is what will allow us to exclude boundary solutions:

\begin{proposition}[{{\cite{iva}}}]
	Over an algebraically closed field, $V(I:J^\infty) = \overline {V(I) - V(J)}$.
	\label{satprop}
\end{proposition}

In our case, we are working in $B_{\mathcal Q} = \Bbbk [x_1,...,x_n]$. It is clear that $x_1x_2...x_n = 0$ exactly when one of the $x_i$ is zero, and so the boundary quasi-steady-states are contained in $V(x_1x_2...x_n)$. Then Proposition \ref{satprop} says that, up to Zariski closure, the zeros of $(I_{\mathcal Q}: (x_1x_2...x_n)^\infty)$ are the nonboundary quasi-steady-states of $I_{\mathcal Q}$.

\begin{theorem}
	Consider a CRN with at-most-bimolecular kinetics and a set of intermediates $\mathcal Q = \{q_1,q_2\}$, such that the corresponding rate law polynomials are neither constant nor both univariate in the same variable. Then $V^a(I_{\mathcal Q}: (x_{q_1}x_{q_2})^\infty)$ is finite, and each zero is solvable over $\Bbbk$.	
	\label{finitethm}
\end{theorem}
\begin{proof}
	For simplicity of notation, we will use $x$ and $y$ instead of $x_{q_1}$ and $x_{q_2}$.
	
	It suffices to find some ideal $J \subseteq (I_{\mathcal Q} : (xy)^\infty)$ such that $V^a(J)$ has finitely many solutions, since $J\subseteq (I_{\mathcal Q} : (xy)^\infty)$ implies $V^a(I_{\mathcal Q} : (xy)^\infty) \subseteq V^a(J)$. Equivalently, it is enough to produce some collection of polynomials in $(I_{\mathcal Q} : (xy)^\infty)$ that have finitely many zeros. In most cases, $I_{\mathcal Q}$ itself will be sufficient.
	
	To show that $(I_{\mathcal Q} : (xy)^\infty)$ is zero-dimensional, it is enough to find two nonzero polynomials $f,g \in (I_{\mathcal Q} : (xy)^\infty)$ such that $g$ is not a zero divisor in $B_{\mathcal Q}/(f)$. Given such polynomials, Krull's Principal Ideal Theorem (Corollary 11.17 in~\cite{atiyah}) implies that $\textrm{codim}((I_{\mathcal Q} : (xy)^\infty)) \geq 2$, and since $\dim(B_\mathcal Q) = 2$, this is implies $\dim(A/(I_{\mathcal Q} : (xy)^\infty)) \leq 0$. This is easiest to accomplish in the case that $f$ is irreducible, as then $B_{\mathcal Q}/(f)$ is an integral domain, so we only need to check that $g \not \in (f)$.
	
	The proof is broken up into the cases: when there is an LCL, when one of the rate law polynomials is linear, and when both rate law polynomials are quadratic. As a general rule, the proofs proceed by carefully analyzing the way in which the generators if $I_{\mathcal Q}$ are insufficient to apply the argument above, and showing that the ways in which they fail to do so imply the existence of other polynomials which are sufficient.
	
	Let the two rate law polynomials in $I_{\mathcal Q}$ be
	\begin{align*}
	\Phi_{x}(x,y) = ax^2  + by^2  + cxy  + dx  + ey  + f\\
	\Phi_{y}(x,y) = a'x^2 + b'y^2 + c'xy + d'x + e'y + f'	.
	\end{align*}
	
\textbf{Case 1: there is an LCL}. Write the LCL as $sx+ry = T$ where $s,r \in \Q$. Since linear polynomials are always irreducible, we would like to show that $sx - ry - T$ does not divide $\Phi_x$.
	
	Both $s$ and $r$ must be nonzero, so the LCL has a root at $(T/s,0)$. Substitute this into $\Phi_x$ and clear the fraction:
$$s^2\Phi_x(T/s,0) = aT^2 + sdT + s^2f.$$

If this is nonzero, then the LCL does not divide $\Phi_x$, and we will be done. Suppose otherwise. Since $s \neq 0$ is rational and $T$ is algebraically independent of $a,d,$ and $f$, it must be the case that $f=d=a = 0$. This means that $\Phi_x = by^2 + cxy + ey = y(by+cx+e)$, and hence $by + cx + e \in (I_{\mathcal Q}:(xy)^\infty)$. This means the following pair of linear equations are in $(I_{\mathcal Q}:(xy)^\infty)$:
\begin{align*}
sx + ry &= T\\
cx + by &= -e.
\end{align*}

This has finitely many solutions - all solvable - if the determinant of the corresponding linear transformation is nonzero. Suppose otherwise:
$$sb - rc = 0.$$

Well, $s$ and $r$ are nonzero, while $b$ and $c$ are algebraically independent or zero, and so this is only possible if $b=c=0$. This means $\Phi_x(x,y) = e$, violating the assumption that the rate law polynomials were nonconstant.

\textbf{Case 2: either $\Phi_x$ or $\Phi_y$ is linear}. Without loss of generality, assume that $\Phi_x$ is linear:
$$\Phi_x(x,y) = dx + ey + f.$$

Further, one of $d$ or $e$ is nonzero, as otherwise $\Phi_x$ would be constant. Assume, without loss of generality, that $d\neq 0$. This means $\Phi_x$ has root $(-f/d,0)$. Since $\Phi_x$ is linear, it is irreducible, so it only remains to check that it does not divide $\Phi_y$. Suppose otherwise, and so $\Phi_y$ must vanish on this root too. Making the substitution and clearing the fraction, we see:
$$d^2\Phi_y(-f/d,0) = a'f^2 - dd'f + d^2f' = 0.$$

Examining the first term, we see that $f=0$ or $a'=0$ - if both were nonzero, $a'f^2$ would be nonzero, hence algebraically independent of all the other terms.

If $f=0$, then $\Phi_x(x,y) = dx + ey$. Note that this means $e\neq 0$. Otherwise $d\in (I_{\mathcal Q}:(xy)^\infty)$, and hence $(I_{\mathcal Q}:(xy)^\infty) = (1)$, so it will have no solutions. We see that $\Phi_x$ has a root at $(e/d,-1)$, and so we will test this root as well in $\Phi_y$:
$$d^2\Phi_y(e/d,-1) = a'e^2 + d^2b' - c'ed - dd'e + e'ed + f'.$$
Since $d,e\neq 0$, we must have $a' = b' = c' = d' = e' = f'
 = 0$ and hence $\Phi_y = 0$, contradicting our assumption that it is nonconstant.

If $a' = 0$, then the above becomes $d^2f' = dd'f$. We may also assume $f\neq 0$, since the other case is covered above. Then it must be that $f' = \alpha f$ and $d' = \beta d$ for some rational numbers $\alpha,\beta$. Comparing the two terms, it is clear that $\alpha =\beta$. Combining these facts, we know
$$\Phi_y(x,y) - \alpha \Phi_x(x,y) = b'y^2 + c'xy = y(b'y + c'x),$$
so $b'y + c'x \in (I_{\mathcal Q}:(xy)^\infty)$. This again reduces us to the case in which two linear polynomials are in $(I_{\mathcal Q}:(xy)^\infty)$, so the zeros of $I_{\mathcal Q}$ are among the solutions of the system of equations
\begin{align*}
dx + ey &= -f\\
c'x + b'y &= 0.
\end{align*}

There are infinitely many solutions only if $b'd = c'e$. Since $d$ is nonzero, this requires $b' = 0$, which means $e=0$ or $c' = 0$. In the first case, $(I_{\mathcal Q}:(xy)^\infty)$ will contain the nonzero element $d$, and so there will be no zeros at all. In the second case, $b' \in (I_{\mathcal Q} : (xy)^\infty)$, so we are done unless $b'=0$. Together, we get $\Phi_y(x,y) = 0$, a contradiction.

\textbf{Case 3: both $\Phi_x$ and $\Phi_y$ have (total) degree 2}. In this case, we immediately rule out $\Phi_x \vert \Phi_y$ - this implies the existence of an LCL, which was treated in Case 1. Hence if either is irreducible, we are done. It remains to see what happens when this is not the case.
	
\textbf{Case 3a: neither polynomial have $x^2$ or $y^2$ terms.} In particular,
$$\Phi_x(x,y) = cxy + dx + ey + f.$$

If this is reducible, then it factors as $c(x + \alpha)(y + \beta)$. By expanding this and comparing terms:
\begin{align*}
\alpha &= d/c\\
\beta  &= e/c\\
\alpha\beta &= f/c.	
\end{align*}

From which we see $de = f$. This is not possible unless $f=0$ and $d=0$ or $e=0$.

Repeating the analysis for $\Phi_y$, we arrive at $f'=0$ and $d'=0$ or $e'=0$.

Suppose that $d=d'=0$, in which case
\begin{alignat*}{3}
\Phi_x(x,y) &= y(cx + e) &\ \ \Rightarrow\ \ & cx + e \in (I_{\mathcal Q}:(xy)^\infty)\\
\Phi_y(x,y) &= y(c'x + e') &\ \ \Rightarrow\ \ & c'x + e' \in (I_{\mathcal Q}:(xy)^\infty)
\end{alignat*}

This will only have infinitely many solutions if $c=c'$ and $e=e'$, but then $\Phi_x = \Phi_y$, and so there is an LCL, which is covered by Case 1.

Now suppose that $d = e' = 0$ instead. This results in:
\begin{alignat*}{3}
\Phi_x(x,y) = y(cx + e)&\ \ \Rightarrow\ \ & cx + e \in (I_{\mathcal Q}:(xy)^\infty)\\
\Phi_y(x,y) = x(c'y + d')&\ \ \Rightarrow\ \ & c'y + d' \in (I_{\mathcal Q}:(xy)^\infty)
\end{alignat*}

Since both $\Phi_x$ and $\Phi_y$ are nonconstant, these two equations are also nonconstant, and so they certainly only have finitely many solutions in $x$ and $y$, all of which are solvable.

The remaining cases, $e=e'=0$ and $d'=e=0$ are symmetric to those already discussed. In each subcase, $(I_{\mathcal Q}:(xy)^\infty)$ has finitely many zeros, all solvable over $\Bbbk$.

\textbf{Case 3b: At least one of $\Phi_x$ or $\Phi_y$ has an $x^2$ or $y^2$ term.} Without loss of generality, suppose that it is $\Phi_x$, and that $a\neq 0$.

This is the most complicated case. We will explore the ways in which Gauss's lemma can fail to show that a rate law polynomial is irreducible, and in each case construct members of $(I_{\mathcal Q}:(xy)^\infty)$ which are enough to show that $V^a((I_{\mathcal Q}:(xy)^\infty))$ that will suffice.

View $\Phi_x$ as a polynomial in $(\Bbbk^a[y])[x]$:

$$\Phi_x(x,y) = ax^2 + (cy + d)x + (by^2 + ey + f).$$

It is primitive, hence irreducible in $B_{\mathcal Q} = \Bbbk^a[x,y]$ if and only if it is irreducible over $\Bbbk^a(y)$. Since it is quadratic, it is irreducible if and only if it has a root in the ground field ($k^a(y)$), which is the case if and only if its discriminant is a square.

Suppose the discriminant is $(\alpha y + \beta)^2$:
$$\mathrm{disc}(\Phi_x) = (cy+d)^2 - 4a(by^2+ey+f) = (\alpha y + \beta)^2$$

Comparing terms:
\begin{align*}
\alpha^2    &= c^2 - 4ab\\
\beta^2     &= d^2 - 4af\\
\alpha\beta &= cd  - 2ae.
\end{align*}

From this, we compute $\alpha^2\beta^2$ in two ways,
\begin{align*}
\alpha^2\beta^2 &= c^2d^2 - 4abd^2 - 4afc^2 + 16a^2bf\\
(\alpha\beta)^2 &= c^2d^2 - 4acde + 4a^2e^2.
\end{align*}

Equate the two and simplifying:
$$bd^2 + fc^2 - 4abf = cde - ae^2$$

(recalling that $a\neq 0$, so we can cancel $-4a$)

Much as before, we make use of algebraic independence. From the $4ae^2$ term, we see that $e=0$. This leaves
$$bd^2 + fc^2 - 4abf = 0.$$

From the last term, $b=0$ or $f=0$. If $b=0$, then $f=0$ or $c=0$, due to the $fc^2$ term. If $f=0$, then $b=0$ or $d=0$.

We can collect these into three cases: $e=b=f=0,e=b=c=0,e=f=d=0$ - note that the $e=b=f=0$ case arises twice above.

\textbf{Case 3bi: $e=b=f=0$}. This lets us write
$$\Phi_x(x,y) = x(ax + cy + d) \Rightarrow ax + cy + d\in (I_{\mathcal Q}:(xy)^\infty).$$

Since $a$ is nonzero, we know that $(-d/a,0)$ is a root. We repeat the trick of substituting this into $\Phi_y(x,y)$:
$$a^2 \Phi_y(-d/a,0) = a'd^2 - add' + a^2f'$$

Since $a\neq 0$, this requires $f' = 0$ and $d = 0$ or $d' = 0$. In the former case, we see $\Phi_x(x,y) = ax + cy$. If $c=0$, then $a\in (I_{\mathcal Q}:(xy)^\infty)$, in which case there is nothing to prove. As such, we assume $c\neq 0$, and note that there is a root at $(-c/a,1)$. Making this substitution in $\Phi_y(x,y)$:
$$a^2\Phi_y(-c/a,1) = c^2a' + a^2b' - acc' - acd' + a^2e' + a^2f'.$$

Since $a,c\neq 0$, this implies $b'=d'=e'=f'=0$, hence $\Phi_y(x,y) = x(a'x + c'y)$, so $a'x + c'y \in (I_{\mathcal Q}:(xy)^\infty)$. We again obtain a pair of linear polynomials in the ideal:
\begin{align*}
ax + cy   &= 0\\
a'x + c'y &= 0.
\end{align*}

This has finitely many solutions unless $ac' = a'c$, which requires $a' = \alpha a$ and $c' = \alpha c$ for some rational $\alpha$. Substituting this back into $\Phi_x,\Phi_y$ reveals $\Phi_y = \alpha \Phi_x$, and hence the existence of an LCL, which was already treated in Case 1.

\textbf{Case 3bii: $e=b=c=0$}. This lets us write
$$\Phi_x(x,y) = ax^2 + dx + f$$

From this, we see that there are only two possible solutions for $x$ (both of which are solvable).

Compare:
$$\Phi_y(x,y) = b'y^2 + (c'x + d')y + (a'x^2 + d'x + f')$$

Now, if $b'\neq 0$, then for any fixed $x \in \Bbbk^a$, this polynomial has just two roots, both solvable over $\Bbbk(x)$, and hence over $\Bbbk$ (as towers of solvable extensions are solvable).

Similarly, if $c'x + d' \neq 0$ for a fixed root $v$ of $\Phi_x$, there are still only finitely many solutions in $y$. Suppose that $c'v + d = 0$. If $c'$ is nonzero, it is algebraically independent of $a,d,d',f$, and hence we arrive at a contradiction unless $v = 0$. This requires $f=0$, hence $\Phi_x(x,y) = x(ax + d)$, so $ax +d \in (I_{\mathcal Q}:(xy)^\infty)$. The only root of this is $-d/a$. Applying analogous reasoning to this root:
$$-c'\frac d a + d' = 0$$
$$c'd = ad'$$
By hypothesis, $a,c' \neq 0$, so it must be the case that $d=d'=0$. This leaves $\Phi_x(x,y) = ax^2$, and so $a \in (I_{\mathcal Q} : (xy)^\infty)$, and we are done. This finally brings us to the case in which $b'=c'=d'=0$. Then both $\Phi_x$ and $\Phi_y$ are univariate in $x$, contradicting the hypothesis that both variables occur among the rate law polynomials.

\textbf{Case 3biii: $e=d=f=0$}. This lets us write
$$\Phi_x(x,y) = ax^2 + by^2 + cxy$$

This case is simple - if it is reducible, then $c=2\sqrt{ab}$. As $a$ is nonzero, this entais $b=c=0$, so $\Phi_x(x,y) = ax^2$, hence $a \in (I_{\mathcal Q} : (xy)^\infty)$, and we are done.

We have seen that, even outside the ``good'' cases, $(I_{\mathcal Q}:(xy)^\infty)$ has only finitely many zeros. In the cases where $I_{\mathcal Q}$ itself has finitely many solutions, we know by Proposition \ref{solvprop} that they are all solvable. When this was not the case, we showed that there are finitely many zeros, which were the roots of linear or quadratic polynomials, hence solvable.
\end{proof}

In fact, for many choices of intermediates, it is not even necessary to remove the boundary steady states by passing from $I_{\mathcal Q}$ to $(I_{\mathcal Q} : (xy)^\infty)$.

\begin{corollary}
	Let $\mathcal Q = \{q_1,q_2\}$ be a set of intermediates in a CRN satisfying the hypotheses of Theorem \ref{finitethm}. If additionally one of the rate law polynomials has a nonzero constant term, then $V^a(I_{\mathcal Q})$ is finite.
	\label{finitethmstar}
\end{corollary}
\begin{proof}
	By applying Theorem \ref{finitethm}, we know that there are finitely many solutions off the boundary, so it suffices to verify that there are only finitely many boundary steady states. For simplicity, let $x = x_{q_1}$, $y = x_{q_2}$ and the rate law polynomials be
	\begin{align*}
	\Phi_{x}(x,y) = ax^2  + by^2  + cxy  + dx  + ey  + f\\
	\Phi_{y}(x,y) = a'x^2 + b'y^2 + c'xy + d'x + e'y + f'	.
	\end{align*}

	Without loss of generality, suppose that $f\neq 0$. Now fix $x=0$:
	$$\Phi_x(0,y) = by^2 + ey + f.$$
	
	This has only finitely many solutions in $y$. Similarly, if we fix $y=0$:
	$$\Phi_x(x,0) = ax^2 +dx + f,$$
	
	which has only finitely many solutions in $x$.
\end{proof}

\begin{corollary}
	For CRNs with at-most-bimolecular kinetics, reduction by QSSA is always possible when there is just one intermediate, or when there are two intermediates satisfying the hypotheses of Corollary \ref{finitethmstar}.
	\label{bcor}
\end{corollary}
\begin{proof}
	For one intermediate, applying the QSSA yields a quadratic or linear univariate equation, which certainly has finitely many solutions, all of which are solvable. For two intermediates, we can apply Corollary \ref{finitethmstar} to see that $V^a(I_{\mathcal Q})$ is finite, with all solutions solvable by radicals.
\end{proof}

Note that the hypothesis of at-most-bimolecular kinetics is not a significant limitation of these results. This hypothesis is necessary to restrict the degrees of the rate law polynomials, so the results still hold in the case that the restriction of the CRN to $\mathcal Q$ is at-most-bimolecular. This is almost always the case. In chemistry, the highest degree that can arise in an elementary mechanism is three, and such reactions are rare, and very slow when they do occur. If such reactions are not crucial to the formation of the final product, they can often be neglected. Even when trimolecular reactions are significant, they will almost certainly cause an accumulation of their reactants. Heuristically, none of those reactants would be good candidates for elimination by QSSA, and so the corresponding degree-three monomial would not be appear in a generator of $I_{\mathcal Q}$.

\subsection{Graph-Theoretical Criteria for Solvability}
Thus far, the criteria given to ensure that reduction by QSSA is possible only apply when there are at most two intermediates. We will now extend these results to larger networks by giving certain structural criteria on the OSR graph.

Recall our previous definition of boundary steady states:
\begin{definition}
	A (quasi) steady-state is a \textbf{boundary} (quasi) steady-state if at least one of its coordinates is zero.
\end{definition}

We might similarly require that it be possible to experimentally achieve the steady states:
\begin{definition}
	A (quasi) steady-state is \textbf{achievable} or \textbf{physically meaningful} if the all of its coordinates are real and nonnegative.
\end{definition}

Our later results will show that reduction by QSSA is possible for the nonboundary, achievable steady states when their OSR graphs are treelike, in a way that will be made more precise later.

\begin{definition}
	Let $G$ be the oriented species-reaction graph corresponding to a CRN $(\mathcal S, \mathcal C, \mathcal R)$. For any subset $\mathcal Q \subseteq \mathcal S$, let $G(\mathcal Q)$ be the induced subgraph of $G$ given by deleting the species vertices not in $\mathcal Q$ (i.e. removing slow species). This is the \textbf{restriction of the OSR graph to $\mathcal Q$}, or, in the context of QSSA, the \textbf{QSSA OSR (QOSR)} graph.
\end{definition}

Within the QOSR graph, we would like to group together species that interact with each other. Such groupings reflect the distribution of variables among the rate law polynomials. This gives rise to a certain equivalence relation. We first recall two pieces of notation:

\begin{nota}
	Given a set $V$ and an equivalence relation $\sim$ on it, for any $v \in V$, we use $[v]$ to denote the equivalence class of $v$. To avoid confusion, this section avoids the notation $[A]$ for the concentration of species A. We use $V/\rel$ to denote the set of equivalence classes of $V$.
\end{nota}

\begin{nota}
	For a graph $G=(V,E)$ and an equivalence relation $\sim$ on $V$, we define $G/\rel$ to be the following graph:
\begin{enumerate}[leftmargin=2cm]
	\item The vertex set is $V/\rel$.
	\item The edge set is $\{([x],[y])\ |\ (x,y) \in E\}$.
\end{enumerate}
\end{nota}

In many cases it may be desirable to suppress the loops created by this identification. However, the equivalence relations we describe below will never cause loops to be formed.

\begin{definition}
	If $\sim$ is an equivalence relation on the vertices of a QOSR graph $G(\mathcal Q)$, we say that it is \textbf{$\mathcal Q$-compatible} when it satisfies the following conditions:
	\begin{enumerate}[leftmargin=2cm]
		\item For all $s\in\mathcal Q$ and $r\in\mathcal R$, $s\not\sim r$.
		\item For all $s,s' \in\mathcal S$ and $r,r'\in\mathcal R$: if $s\rightarrow r$ and $s' \rightarrow r'$ are edges, then $s \sim s'$ if and only if $r \sim r'$.
	\end{enumerate}
\end{definition}

The first condition ensures that species and reactions remain separate when passing from $G$ to $G/\rel$. This means that $\sim$ restricted to $\mathcal Q$ remains an equivalence relation. The second ensures that each equivalence class is closed with respect to species which react together; e.g., if $A+B \rightarrow C$ is in the CRN, then we are guaranteed that $A \sim B$. The converse does not necessarily hold, as it could be the case that $A$ and $B$ do not react, but there are reactions $r,r'$ such that $A\rightarrow r$ and $B\rightarrow r'$ are in the QOSR graph, and $r\sim r'$.

The motivation for the following theorem is as follows: if the QOSR graph is approximately treelike, we can start at a root of the tree, solve for the associated concentrations, and then substitute those values into the remaining equations. Substituting solvable roots into solvable roots should result in another solvable root.

\begin{theorem}
	Let $G$ be the OSR graph of a CRN, $\mathcal Q$ a set of intermediates such that $LCL(\mathcal R,\mathcal Q)$ is empty, and $H=G(\mathcal Q)$ the associated QOSR graph. Suppose there exists a $\mathcal Q$-compatible relation on $H$ such that
	\begin{enumerate}[leftmargin=2cm]
		\item $H/\rel$ has no directed cycles,
		\item For every equivalence class of species $[q] \in \mathcal Q/\rel$ (i.e. a ``species'' vertex in $H/\rel$), reduction of the original CRN by QSSA is possible when $[q]$ is the set of intermediates,
		\item For each equivalence class of species $[q]$, specializing the constant terms of the rate law polynomials among the generators of $I_{[q]}$ to values algebraic over $\Bbbk_{\mathcal Q - [q]}$ does not make $V^a(I_{[q]})$ infinite.
	\end{enumerate}
	Then reduction by QSSA is possible for $\mathcal Q$.
	\label{treethm}
\end{theorem}
\begin{proof}
	We proceed by induction on the number of equivalence classes of $\mathcal Q$. If there is just one equivalence class, then it is all of $\mathcal Q$, and so QSSA reduction is possible.
	
	Since $H/\rel$ has no directed cycles, pick some equivalence class $[q]\in\mathcal Q/\rel$ such that there is no directed path from $[q]$ to any other member of $\mathcal Q/\rel$ in $H/\rel$.
	
	By hypothesis, QSSA reduction is possible for $I_{\mathcal Q - [q]}$. In fact, this reduction is possible over the field $\Bbbk_{\mathcal Q}$ (a subfield of $\Bbbk_{\mathcal Q - [q]}$). This is because $\mathcal Q$-compatibility ensures the concentrations corresponding to members of $[q]$ do not appear as coefficients of the generators of $I_{\mathcal Q - [q]}$, which means they do not appear when computing a Gr\"obner basis for $I_{\mathcal Q - [q]}$. This means the concentrations in any quasi-state-state for $\mathcal Q - [q]$ are also solvable over $\Bbbk_{\mathcal Q}$.
	
	Let $(\alpha_1,...,\alpha_m)$ be such a zero of $I_{\mathcal Q - [q]}$; we can view it as a ``partial solution'', in that it will be extended to a zero of $I_{\mathcal Q}$. This is similar to the Extension Theorem~\cite{iva}, but we need to be careful to ensure that solvability is preserved.
	
	We begin by following the algorithm outlined in Proposition \ref{algprop}. First, a Gr\"obner basis of $I_{[q]}$ (in the ring $\Bbbk_{[q]}[x_r\ |\ r\in [q]]$) is computed with respect to some lexicographic ordering. This provides univariate polynomials $f_r(x_r)$ in each of the variables $x_r$, which are solvable by the second hypothesis. Some of the parameters in $\Bbbk_{[q]}$ are concentrations of species in $\mathcal Q - [q]$. Specialize these to the values of the partial solution above. By Lemma \ref{speclemma}, this preserves the solvability of the polynomials. Furthermore, this specialization preserves finiteness, as required by hypothesis three.
	
	As such, we obtain a collection of polynomials $f'_r(x_r)$ after the substitution. For any zero $(\beta_1,...,\beta_n)$ of the specialized $I_{[q]}$, each coordinate vanishes on the corresponding $f'_r(x_r)$, which means that it is the root of a solvable polynomial over $\Bbbk(\alpha_1,...,\alpha_m)$. 
	
	Finally, we know that $(\alpha_1,...,\alpha_m,\beta_1,...,\beta_n$) is a zero of $I_{\mathcal Q}$. Several of the generators of $I_{\mathcal Q}$ coincide with generators of $I_{\mathcal Q - [q]}$, so substitution of the $\alpha_i$'s eliminates them. Because there are no LCLs, remaining are generators of $I_{[q]}$, and by construction these will vanish on the $\beta_j$'s after substitution of the $\alpha_i$'s into their constant terms.
	
	These zeros are solvable in each coordinate, and we only gain finitely many more in extending from $I_{\mathcal Q - [q]}$ to $I_{\mathcal Q}$. Any zero of $I_{\mathcal Q}$ can be obtained in this way as well; call such a zero $(\alpha_1,...,\alpha_m,\beta_1,...,\beta_n)$ where the $\alpha$ coordinates correspond to concentrations from $\mathcal Q - [q]$ and the $\beta$ coordinates to those from $[q]$. The generators of $I_{\mathcal Q - [q]}$ are also generators of $I_{\mathcal Q}$, so they must vanish on $(\alpha_1,...,\alpha_m,\beta_1,...,\beta_n)$ and hence on $(\alpha_1,...,\alpha_m)$ alone (as they do not have variables into which the $\beta$s could be substituted). Likewise, generators of $I_{[q]}$ are also among the generators of $I_{\mathcal Q - [q]}$, and must vanish on $(\beta_1,...,\beta_m)$ after the substitution (``specialization'') of the other concentrations to the $\alpha_i$s.
	
	In summary, we can extend any (solvable) zero of $I_{\mathcal Q - [q]}$ to a (solvable) zero of $I_{\mathcal Q}$ by specializing $I_{\mathcal [q]}$. This produces only finitely many zeros. Any zero of $I_{\mathcal Q}$ can be obtained in this way, and so $I_{\mathcal Q}$ has only finitely many zeros, all of which are solvable over $\Bbbk$, which means reduction by QSSA is possible.
\end{proof}

The hypotheses of Theorem \ref{treethm} exclude any choice of intermediates $\mathcal Q$ for which there exists some linear conservation law. In such cases, it is possible for the species constrained by an LCL to be spread across several equivalence classes. The theorem fails in such cases because the argument that no concentration from $[q]$ appears in the generators of $I_{\mathcal Q - [q]}$ may be false.

Note that Theorem \ref{treethm} remains true if using weaker definitions of ``QSSA reduction is possible'', such as restricting to some combination of nonboundary, nonzero, nonnegative, or real-valued concentrations. However, this requires the additional step of removing spurious solutions, such as those with negative concentrations, at each step, and so it is necessary to verify that there is at least one solution in such cases.

At first glance, the conditions of Theorem \ref{treethm} seem difficult to satisfy, particularly the ``niceness'' condition of specialization required by the third hypothesis. Nevertheless, we will now show that, when we restrict to nonboundary quasi-state-states, Theorem \ref{treethm} will apply if each equivalence class satisfies Theorem \ref{finitethm} and has at least one zero. The last condition follows naturally from the motivation for Theorem \ref{treethm}: in order to make a substitution, we need something in which to make a substitution.

\begin{corollary}
	For some CRN, let $\mathcal Q$ be a set of intermediates and $G$ its OSR graph. If there exists a $\mathcal Q$-compatible relation on $G(\mathcal Q)$ for which $[q]\in\mathcal Q/\rel$ satisfies Theorem \ref{finitethm}, and for each equivalence class $[q]$, $I_{[q]}$ has at least one nonboundary zero, then there are finitely many physically meaningful nonboundary quasi-steady-states, all expressible by radicals.
	\label{osrcor}
\end{corollary}
\begin{proof}
	The substitutions made in Theorem \ref{treethm} preserve the independence relationships among the coefficients of the rate law polynomials, and hence the reasoning with algebraic independence made in Theorem \ref{finitethm} continues to apply. If we can provide a suitable interpretation for what it means to have a linear conservation law after substitution, we will be done.
	
	Let $\Phi_x',\Phi_y'$ be a pair of rate law polynomials with a substitution made in the constant term. A linear conservation law is a dependence of the form
	$$\Phi_x' + \alpha \Phi_y' = 0,$$
	
	where $\alpha \in \Q$. This implies that there is also an LCL among their nonconstant terms.	Since no substitution was made into those coefficients, this lifts back to the unsubstituted rate law polynomials. As such, either $\Phi_x - \alpha\Phi_y = 0$, implying the existence of the necessary LCL in $I_{\mathcal Q}$, or $\Phi_x - \alpha\Phi_y$ is a nonzero constant, contradicting the hypothesis that $I_{\mathcal Q}$ has at least one solution.
\end{proof}

As we will demonstrate below in Example 35, the restriction to nonboundary steady states in Corollary \ref{osrcor} is necessary.


\section{Examples}
First, we will discuss the CRN presented by Pantea, Gupta, Rawlings, and Craciun in the context of what we have proven. Following this, we will provide some examples of CRNs with interesting properties. These examples involve reducible univariate polynomials and Galois groups such as $S_4\times \Z_2$, neither of which we would typically expect from polynomials with ``generic-looking'' coefficients. We conclude with an example of a mechanism which demonstrates the necessity of the restriction to nonboundary steady states in Corollary \ref{osrcor}.

For several examples, Maple code is available which demonstrates the computation. We have implemented a version of the algorithm outlined in Proposition \ref{algprop}, which computes the Galois groups associated to a mechanism, given the QSSA ideal. Several examples also work through the algorithm manually to help clarify the details of its use.

\begin{example}
\label{exorig}
	This is the original insolvable mechanism~\cite{pantea}:
	
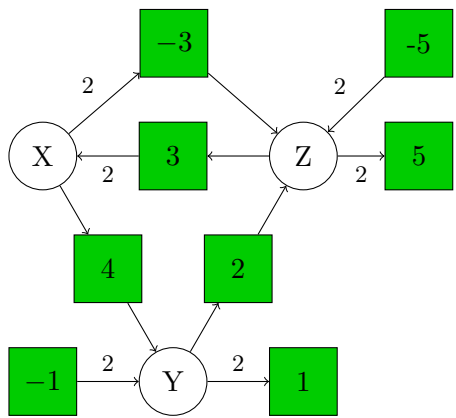
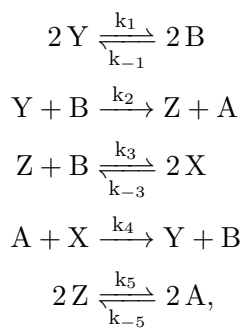
\begin{figure}
	\begin{center}
	\begin{tikzpicture}
		\node[circle,draw,minimum size=.9cm] at (0,3) (X) {X};
		\node[circle,draw,minimum size=.9cm] at (1.732,0) (Y) {Y};
		\node[circle,draw,minimum size=.9cm] at (3.464,3) (Z) {Z};
		
		\node[rectangle,draw,minimum size=.9cm,fill=rxngreen] at (0,0) (r-1) {$-1$};
		\node[rectangle,draw,minimum size=.9cm,fill=rxngreen] at (3.464,0) (r1) {1};
		\node[rectangle,draw,minimum size=.9cm,fill=rxngreen] at (2.598,1.5) (r2) {2};
		\node[rectangle,draw,minimum size=.9cm,fill=rxngreen] at (1.732,3) (r3) {3};
		\node[rectangle,draw,minimum size=.9cm,fill=rxngreen] at (1.743,4.5) (r-3) {$-3$};
		\node[rectangle,draw,minimum size=.9cm,fill=rxngreen] at (0.866,1.5) (r4) {4};
		\node[rectangle,draw,minimum size=.9cm,fill=rxngreen] at (5,3) (r5) {5};
		\node[rectangle,draw,minimum size=.9cm,fill=rxngreen] at (5,4.5) (r-5) {-5};
		
		\draw[->] (X) -- (r4);
		\draw[->] (X) -- node[above left] {\footnotesize 2} (r-3);
		
		\draw[->] (Y) -- node[above]  {\footnotesize 2} (r1);
		\draw[->] (Y) -- (r2);
		
		\draw[->] (Z) -- node[below]  {\footnotesize 2} (r5);
		\draw[->] (Z) -- (r3);
		
		\draw[->] (r-1) -- node[above]  {\footnotesize 2} (Y);
		\draw[->] (r2) -- (Z);
		\draw[->] (r3) -- node[below]  {\footnotesize 2} (X);
		\draw[->] (r-3) -- (Z);
		\draw[->] (r4) -- (Y);
		
		\draw[->] (r-5) -- node[above left] {\footnotesize 2} (Z);
	\end{tikzpicture}
	\end{center}
\caption{\label{fig:pantea}QOSR graph for the mechanism in Example \ref{exorig}, taken from Pantea et. al~\cite{pantea}.}
\end{figure}

\begin{align*}
	\ce{2Y &<=>[k_1][k_{-1}]2B}\\
	\ce{Y + B &->[k_2] Z + A}\\
	\ce{Z + B &<=>[k_3][k_{-3}] 2X}\\
	\ce{A + X &->[k_4] Y + B}\\
	\ce{2Z &<=>[k_5][k_{-5}] 2A},
\end{align*}
for which the chosen intermediates are $\mathcal Q = \{X,Y,Z\}$. The QOSR graph is displayed in Figure \ref{fig:pantea}

The CRN is at-most-bimolecular, and there are 3 intermediates. The univariate polynomial in $z = [Z]$ obtained by computing a Gr\"obner basis with respect to
$$x > y > z,$$
is degree $8$, demonstrating the sharpness of the bound used in the proof of Proposition \ref{solvprop}. The corresponding Galois group is isomorphic to $S_8$, which is insolvable. Corollary \ref{bcor} shows that this example has the minimum number of intermediates necessary to have an insolvable Galois group, as for two intermediates we can almost always guarantee solvability.

Furthermore, five total species is chemically reasonable when there are three intermediates, as we expect at least one ``overall reactant'' and one ``overall product'', neither of which is an intermediate. Additionally, we can see that the QOSR graph in Figure \ref{fig:pantea} has a directed cycle among $X$, $Y$, and $Z$. The only way to remove a cycle with a $\mathcal Q$-compatible equivalence relation is to place all its members in one equivalence class, which prevents the application of Corollary \ref{osrcor}.
\end{example}

\begin{example}
\label{panteatweakex}
We can break the cycle from Example \ref{exorig} by deleting the edge $(2,Z)$. This is the same as changing the reaction

\begin{figure}
	\begin{center}
	\begin{tikzpicture}

		\draw[rectangle, fill=teal, opacity =0.25] (-0.65,2.35) -- (5.65,2.35) -- (5.65,5.15) -- (-0.65,5.15) -- (-0.65,2.35);
		
		\draw[rectangle,fill=orange, opacity=0.25] (-0.65,-0.65) -- (4.114,-0.65) -- (4.114, 2.15) -- (-0.65,2.15) -- (-0.65,-0.65);

		\node[circle,draw,minimum size=.9cm] at (0,3) (X) {X};
		\node[circle,draw,minimum size=.9cm] at (1.732,0) (Y) {Y};
		\node[circle,draw,minimum size=.9cm] at (3.464,3) (Z) {Z};
		
		\node[rectangle,draw,minimum size=.9cm,fill=rxngreen] at (0,0) (r-1) {$-1$};
		\node[rectangle,draw,minimum size=.9cm,fill=rxngreen] at (3.464,0) (r1) {1};
		\node[rectangle,draw,minimum size=.9cm,fill=rxngreen] at (2.598,1.5) (r2) {2};
		\node[rectangle,draw,minimum size=.9cm,fill=rxngreen] at (1.732,3) (r3) {3};
		\node[rectangle,draw,minimum size=.9cm,fill=rxngreen] at (1.743,4.5) (r-3) {$-3$};
		\node[rectangle,draw,minimum size=.9cm,fill=rxngreen] at (0.866,1.5) (r4) {4};
		\node[rectangle,draw,minimum size=.9cm,fill=rxngreen] at (5,3) (r5) {5};
		
		\draw[->] (X) -- (r4);
		\draw[->] (X) -- node[above left] {\footnotesize 2} (r-3);
		
		\draw[->] (Y) -- node[above]  {\footnotesize 2} (r1);
		\draw[->] (Y) -- (r2);
		
		\draw[->] (Z) -- node[below]  {\footnotesize 2} (r5);
		\draw[->] (Z) -- (r3);
		
		\draw[->] (r-1) -- node[above]  {\footnotesize 2} (Y);
		\draw[->,dashed,thick,color=red] (r2) -- (Z);
		\draw[->] (r3) -- node[below]  {\footnotesize 2} (X);
		\draw[->] (r-3) -- (Z);
		\draw[->] (r4) -- (Y);
	\end{tikzpicture}
	\end{center}
\caption{\label{fig:saturation}QOSR graph for the modified Pantea mechanism of Example \ref{panteatweakex}. The deleted edge is dashed and marked in red, and regions are drawn around the graph to denote equivalence classes used in applying Theorem \ref{treethm}.}
\end{figure}
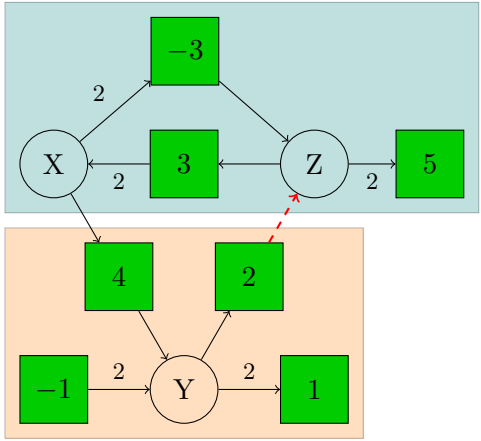

\begin{center}
	\ce{Y + B ->[k_2] Z + A}\ \ into\ \ \ce{Y + B ->[k_2]  A}.
\end{center}
To simplify the calculations and illustrate an interesting point about boundary solutions, we also remove the backwards reaction, labeled by $-5$. We will continue to use $\mathcal Q = \{X,Y,Z\}$.

Let $x$, $y$, and $z$ be the concentrations of $X$, $Y$, and $Z$, respectively, and likewise, $a$ and $b$ the concentrations of $A$ and $B$. The rate equations for the intermediates are:
\begin{align*}
	\frac{dx}{dt} &= -2k_{-3}x^2 - k_4ax + 2k_3bz\\
	\frac{dy}{dt} &= -2k_{1}y^2 - k_2by + 2k_{-1}b^2 + k_4ax\\
	\frac{dz}{dt} &= -2k_5z^2 - k_3bz + k_{-3}x^2 
\end{align*}

The QOSR graph of the resulting CRN is displayed in Figure \ref{fig:saturation}. We would like to show that reduction by QSSA is possible. Since there are three intermediates, this will require Corollary \ref{osrcor}. 

First, we need a $\mathcal Q$-compatible equivalence relation. These are indicated by the shaded regions in Figure \ref{fig:saturation}. The equivalence classes of reactants are:
$$\{X,Z\}\ \textrm{and}\ \{Y\},$$

which induces the following equivalence classes of reactions:
$$\{-3,3,5\},\ \{4\},\ \{-1\},\textrm{ and }\ \{2,1\}.$$

The hypotheses of Corollary \ref{osrcor} are satisfied: both equivalence classes have fewer than two intermediates, the kinetics are at-most-bimolecular, and the rate law polynomials are of the appropriate form. That each equivalence class has at least one quasi-steady-state is less immediate, but is readily verifiable by a computer algebra package. This ensures that reduction by QSSA is at least possible on the nonboundary solutions.

However, even for the boundary solutions, we can verify that reduction by QSSA is possible. By inspection, the only boundary steady state occurs when $x=z=0$. This determines $y$ as one of the two roots of a quadratic equations. In combination with the above, this ensures that reduction by QSSA is always possible with $\mathcal Q$.

Knowing this, we may as well work with $I_{\mathcal Q}$. This gives rise to more interesting Galois groups than might be expected.

There are no LCLs among the intermediates, so the constants are just $a$, $b$, and the rate constants, hence:
$$\Bbbk_{\mathcal Q} = \Q(a,b,k_1,k_{-1},k_2,k_3,k_{-3},k_4,k_5)\ \ \ \textrm{and}\ \ \ A_{\mathcal Q} = \Bbbk_{\mathcal Q} [x,y,z],$$
with the following QSSA ideal:
$$I_{\mathcal Q} = \langle -2k_{-3}x^2 - k_4ax + 2k_3bz, -2k_{1}y^2 - k_2by +2k_{-1}b^2 + k_4ax, -2k_5z^2 - k_3bz + k_{-3}x^2 + 2k_{-5}a^2\rangle.$$

The corresponding Galois groups of the nonboundary quasi-steady-states solvable by the previous remarks, and we can follow the algorithm outlined in Proposition \ref{algprop} to calculate them. The Galois groups (over $\Bbbk$) for the irreducible factors of the univariate polynomials in each coordinate are as follows:
\begin{center}
	\begin{tabular}{c|c}
	var. & Galois group\\\hline\hline
	$x$ & $S_3$ or $\{e\}$\\\hline
	$y$ & $S_4\times \Z_2$ or $\Z_2$\\\hline
	$z$ & $S_3$ or $\{e\}$
\end{tabular}
\end{center}

There are two Galois groups in each case because the polynomials obtained from the Gr\"obner bases factor into the product of two irreducibles.

Recall that we found that reduction by QSSA is possible even on the boundary by showing that on the boundary, $x=z=0$ and the possible concentrations of $y$ were roots of a quadratic. Inspecting the above, this is consistent with the appearance of $\Z_2$ and $\{e\}$ as Galois groups.

To verify our predictions, we can check that those Galois are lost when we work in $(I_{\mathcal Q} : (xyz)^\infty)$. First, the saturation needs to be computed:
$$I_{\mathcal Q}' = (I_{\mathcal Q} : \langle xyz\rangle^{\infty})$$
\begin{align*}
	I'_{\mathcal Q} =& \langle ak_4x+4k_5z^2, -ak_4x+2bk_3z-2k_{-3}x^2,  abk_3k_4+2ak_4k_5z+4k_{-3}k_5xz,\\
	&  a^2k_4^2k_5x+ab^2k_3^2k_4+4ak_{-3}k_4k_5x^2+4k_{-3}^2k_5x^3, -ak_4x-2b^2k_{-1}+bk_2y+2k_1y^2\rangle
\end{align*}

While the generating set is now much more complicated, the boundary solutions have been eliminated. When we check the Galois groups again for $I'_{\mathcal Q}$, only $S_3$ and $S_4\times \Z_2$ remain. As remarked previously, using a boundary solution to perform the reduction step is equivalent to removing a species entirely, which is generally not desirable. After saturation, only the Galois groups and roots which are most relevant to performing the reduction remain.
\end{example}

Finally, we present an example of a mechanism which shows that Corollary \ref{osrcor} cannot be generalized to include boundary quasi-steady-states.

\begin{example}
	Consider the following CRN:
	
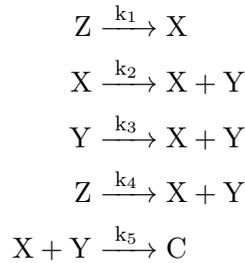
\begin{align*}
	\ce{Z &->[k_1] X}\\
	\ce{X &->[k_2] X + Y}\\
	\ce{Y &->[k_3] X + Y}\\
	\ce{Z &->[k_4] X + Y}\\
	\ce{X + Y &->[k_5] C}
\end{align*}
	
	We choose intermediates $\mathcal Q = \{X,Y,Z\}$, with concentrations $x$, $y$, and $z$, respectively. Note that this is already a poor choice of intermediates, as it is clear that it will only be at (quasi) steady state when $z=0$. The rate law polynomials are:
\begin{align*}
	\Phi_x(x,y,z) &= -k_5xy + k_1z + k_4z\\
	\Phi_y(x,y,z) &= -k_5xy + k_2x + k_4z\\
	\Phi_z(x,y,z) &= -k_1z  - k_4z
\end{align*}
	
	We could try to apply Corollary \ref{osrcor}, using a $\mathcal Q$-compatible equivalence relation with the species separated into $\{X,Y\}$ and $\{Z\}$. By inspection, QSSA is possible for each of those equivalence classe. However, the only zero of $I_{\{Z\}}$ is $z=0$, a boundary solution, so the hypotheses of the corollary are not satisfied. 
	
	When we make this substitution into the other rate law polynomials, it is clear that QSSA will not be possible:
\begin{align*}
	\Phi_x(x,y,0) &= -k_5xy\\
	\Phi_y(x,y,0) &= -k_5xy + k_2x
\end{align*}

This system has infinitely many solutions, and so reduction by QSSA is not possible.
\end{example}


\section{Discussion}
This paper provides partial answers to several questions raised in the original paper on the relationship between insolvability and QSSA~\cite{pantea}. Corollary \ref{bcor} shows that both examples of nonsolvable mechanisms presented in that paper are minimal with respect to the number of intermediates (given at-most-bimolecular kinetics). Furthermore, Theorem \ref{treethm} shows that they are also minimal with respect to the structure of the QOSR graph as well: the mechanisms have a cycle among three intermediates. However, the mechanisms are not necessarily minimal with respect to the Galois group, as $A_5$ is the smallest insolvable group, while the mechanisms presented by Pantea, Gupta, Rawlings, and Craciun all had Galois groups isomorphic to $S_8$. In this case, the B\'ezout bound is sharp and the Galois groups of each degree $n$ polynomial is $S_n$. Neither pattern holds in general: we have discovered examples in which the B\'ezout bound is not sharp, and also examples in which it is sharp, but the resulting degree $8$ polynomial has Galois group $S_7$ rather than $S_8$.

Our results also begin to answer the question of why, historically, it was not noticed that reduction by QSSA could fail due to the absence of solutions in radicals. First we saw that reduction by QSSA is typically possible for small number of intermediates. For many mechanisms, the number of intermediates is not too large, especially when the computations are performed by hand. Extending this, we showed that reduction by QSSA is possible for larger CRNs satisfying certain structural constraints. These results cover many well-known mechanisms in chemistry, biology, and chemical engineering, particularly for catalysis and enzyme-substrate interactions~\cite{fogler}.

In fact, our results suggest that finiteness of the number of steady states is in some ways more critical than solvability. For all of our results, we require finiteness as the first step to showing that the solution is expressible in radicals. Its failure to be preserved under specialization makes it difficult to extend solvability among subsets of the intermediates to the entire system. There are several interesting questions that this observation suggests:
\begin{enumerate}[leftmargin=2cm]
	\item Does Theorem \ref{finitethm} generalize to larger sets of intermediates?
	\item How important is excluding boundary solutions to finiteness?
	\item Does there exist some analog of Theorem \ref{treethm} for finiteness?
\end{enumerate}

One promising angle of attack for Question 1 would be to generalize Proposition \ref{indprop} or Corollary \ref{indcor} by replacing divisibility with ideal membership. That is, we would like it to be true that for rate law polynomials $f,g_1,...,g_n$, the relationship $f \in \langle g_1,...,g_n\rangle$ implies that $f = \sum c_i g_i$, for $c_i \in \Q$, and hence that there exists an LCL. As noted previously, such an LCL helps to compensate for the degeneracy in the system which arises when one rate law polynomial can be expressed in terms of the others. Of course, the general method of proof in Theorem \ref{finitethm} quickly becomes impractical for large sets of intermediates.

Another area of interest is the relationship between other properties of the CRN and its associated Galois groups. At the moment, the computation of the Galois groups is often the most time-consuming part of performing the algorithm in Proposition \ref{algprop}, and it is not even possible to do so on a computer beyond degree eight. A firmer understanding of the Galois groups which arise from CRNs (without specialization) may allow this computation to be avoided or accelerated. This suggests several questions:
\begin{enumerate}[leftmargin=2cm]
	\item Are there more properties of a CRN's Galois groups which can be related to the OSR or other graphs?
	\item In a fixed CRN, what relationships exist among the Galois groups corresponding to different species? What about extensions of the CRN?
	\item Is it possible to determine some Galois groups directly from the structure of graphs associated to the CRN? (The OSR graph seems a good candidate, but it could be some other graph).
	\item For a group $G$, does there exist a CRN and choice of intermediates such that $G$ is a Galois group associated to one of the intermediates?
\end{enumerate}
The first and second questions seem the most approachable. Furthermore, progress on these two may allow some information to be obtained about the Galois groups, even the computation cannot be performed. This is exactly what Corollary \ref{osrcor} provides: as long as the set of intermediates is extended in a tree-like way, the solvability is preserved for the Galois groups of the positive real quasi-state-states.

The third question is likely very difficult to answer, but it would be extremely useful. At the moment, the computation of the Galois groups is the slowest step in the algorithm outlined earlier. Further, we know of no software which can compute Galois groups over fields of rational functions in $\Q$ for polynomials of degree greater than $8$. This bound is hit very quickly, even for few intermediates and at-most-bimolecular kinetics. It may be possible that there is enough extra structure in CRNs to enable progress on Question 3.

The fourth question is perhaps the most mathematically interesting of the four. While of somewhat less practical use, its answer would likely necessitate progress on the other three. Even without restricting to at-most-bimolecular kinetics, we suspect that not all groups are realizable by QSSA. For instance, we know of no CRNs in which a nontrivial Galois group of odd order arises, and conjecture that there are none.

\section*{Acknowledgements}
I would like to thank Dr. Anne Shiu for her mentorship and many valuable insights. I am also grateful to Ola Sobieska for helpful conversations. This research was conducted as part of Texas A\&M University's Mathematics REU on Algebraic Methods in Computational Biology, funded by the National Science Foundation through REU grant DMS-1460766.

\bibliography{refs}
\end{document}